\documentclass[12pt, a4paper]{amsart}
\usepackage{amsmath,amssymb,amscd,amsfonts}
\usepackage{ifthen}
\usepackage[T1]{fontenc}
\usepackage[utf8]{inputenc}
\usepackage[all]{xy}
\usepackage{graphicx}
\usepackage{enumerate}
\usepackage{xspace}
\usepackage{pdfsync}
\usepackage{epic}
\usepackage{dsfont}

\newtheorem{theorem}{Theorem}
\newtheorem{remark}[theorem]{Remark}
\newtheorem{definition}[theorem]{Definition}
\newtheorem{lemma}[theorem]{Lemma}

\newtheorem{proposition}[theorem]{Proposition}
\newtheorem{conjecture}[theorem]{Conjecture}
\newtheorem{corollary}[theorem]{Corollary}
\newtheorem{example}[theorem]{Example}

\begin{document}
\title[]{Kodaira additivity, birational isotriviality and specialness. }

\author{Fr\'ed\'eric Campana}
\address{Universit\'e de Lorraine \\
 Institut Elie Cartan\\
Nancy, France}

\email{frederic.campana@univ-lorraine.fr}



\date{\today}

\begin{abstract} We show that a smooth projective fibration $f:X\to Y$ between connected complex quasi-projective manifolds satisfies the equality $\overline{\kappa}(X)=\kappa(X_y)+\overline{\kappa}(Y)$ of Logarithmic Kodaira dimensions if its fibres $X_y$ have semi-ample canonical bundles. Without the semi-ampleness assumption, this additivity was conjectured in \cite{Pop14}. Several cases are established in \cite{PS22}, which inspired the present text. Although the present results overlap with those of \cite{PS22} in the projective case, the approach here is different, based on the r\^ole played by birationally isotrivial fibrations, special manifolds and the core map of $Y$ introduced and constructed in \cite{Ca07}. 
\end{abstract}

\maketitle

\tableofcontents

\section{Introduction}

The base field is $\mathbb{C}$, the field of complex numbers. For $X$ quasi-projective $X$, we define $\overline{\kappa}(X)$ in the usual way: $\overline{\kappa}(X)=\kappa(\overline{X},K_{\overline{X}}+D)$ if $\overline{X}$ is any projective smooth `good' compactification of $X$, that is: such that $D=\overline{X}\setminus X$ is a divisor of simple normal crossings in $\overline{X}$. 

\begin{theorem}\label{+} Let $f:X\to Y$ be a submersive projective holomorphic map between two connected complex quasi-projective manifolds $X$ and $Y$. Assume that its fibres $X_y$ have semi-ample canonical bundles.

Then `addivity' holds, which means that: $\overline{\kappa}(X)=\kappa(X_y)+\overline{\kappa}(Y)$.

In particular, $\overline{\kappa}(X)=-\infty$ if $\overline{\kappa}(Y)=-\infty$. \end{theorem}


\begin{remark}1. If one applies \cite{T21}, Theorem 1.1 by B. Taji, instead of \cite{AC}, the very same proof works with the weaker hypothesis that the $X_y's$ have good minimal models, instead of semi-ample canonical bundle. The abundance conjecture then implies \cite{Pop14}, Conjecture 3.4. 

2. In \cite{PS22}, Theorems C and H, the authors give a different proof in the projective case, assuming other versions of the abundance conjecture.

3.  Theorem A of \cite{PS22} says that if $f:X\to Y$ is a smooth fibration between projective manifolds such that $\kappa(X_y)\geq 0$ and $\kappa(X)=\kappa(X_y)+dim(Y)$, then $\kappa(Y)=dim(Y)$. Theorem \ref{+} implies this, but only under the much stronger semi-ampleness hypothesis of $K_{X_y}$.

4. By \cite{Kaw}, the $C_{n,m}$ inequality is true when the generic fibre has good minimal models, without the smoothness hypothesis. 
\end{remark}

The proof of Theorem \ref{+} is an easy combination of the following steps, which we describe here when $X,Y$ are projective ( the quasi-projective case is proved along the same lines in \S\ref{open}):

\medskip

1. The additivity is valid when $f$ is birationally isotrivial, , which means that there exists a generically finite base-change $v:Y_1\to Y$ such that $f_1:X_1\to Y_1$ is birational to $F\times Y_1$ over $Y_1$, for some $F$, birational to the fibres $X_y$ of $f$. See \S2 for the proof.

\medskip

The next steps depend on `Orbifold geometry' notions (`specialness') and results (the `core map'), originating in \cite{Ca07} and recalled in \S3.

\medskip

2. By \cite{AC}\footnote{As already said, when $X_y$ has a good minimal model, we could invoke \cite{T21} instead, without changing anything to the proof.}, if $Y$ is `special', then $f$ is birationally isotrivial.

\medskip

3. We then apply the `core map' $c_Y:Y\to (C,\Delta_{c_Y})$, which has special fibres and general type `orbifold base' $(C,\Delta_{c_Y})$.

\medskip

4. The first and second steps applied to $f_c:X_c\to Y_c$, the restriction of $f$ over a general point $c\in C$, show that $\kappa(X_c)=\kappa(X_y)+\kappa(Y_c)$. Indeed: $Y_c$ is special, and $f_c$ is thus birationally isotrivial. 

\medskip

5. The `orbifold' version of Viehweg's additivity theorem (\cite{V}, \cite{Ca07}), according to which the additivity for Kodaira dimensions is valid when the `orbifold base' $(C,\Delta_g)$ of the fibration $g:Z\to C$ is of general type, whatever the fibres, implies that: 

$\kappa(Y)=\kappa(Y_c)+dim(C)$, and:

$\kappa(X)=\kappa(X_c)+dim(C)$, since $\kappa(C,\Delta_{c_Y\circ f})=dim(C)$ (Lemma \ref{>}).

6. Combining steps 4 and 5, we get: 

$\kappa(X)-\kappa(Y)=\kappa(X_c)-\kappa(Y_c)=\kappa(X_y)$, as claimed.

\medskip

 {\bf Acknowledgement:} I heartily thank M. P\u aun for his careful reading of the text.

 \section{Specialness and  birational isotriviality.}
 
 \medskip
 
{\bf Convention:} A projective fibration $f:X\to Y$ is a projective surjective morphism with connected fibres between smooth connected quasi-projective manifolds $X$ and $Y$. A fibre $X_y=f^{-1}(y)$ of $f$ is `general' if $y$ is so in $Y$, that is: if $y$ lies outside countably many suitable Zariski closed strict subsets of $Y$.

 \subsection{The birationally isotrivial case.}\label{birisot}

 \begin{definition} Let $f:X\to Y$ be a projective fibration between smooth connected quasi-projective $X,Y$. We say that $f$ is birationally isotrivial if there exists some finite morphism $v: Y'\to Y$ such that, if $f':X'\to Y'$ is deduced from $f$ by the base-change $v$, then $X'$ is birational to $F\times Y'$ over $Y'$. Equivalently, this means that all  generic $X_y$ are birational to some fixed $F$. 
  \end{definition}

{\bf Notation:} Let $(\overline{X},D)$ be a good compactification of a quasi-projective smooth $X$. We then write $\overline{K}_X:=K_{\overline{X}}+D$ so that $\overline{\kappa}(X)=\kappa(\overline{X},\overline{K}_X)$. If $f:X\to Y$ is a projective fibration between smooth quasi-projective $X,Y$, which is the restriction of some fibration $f:\overline{X}\to \overline{Y}$, where $(\overline{X},D)$ and $(\overline{Y},D_Y)$ are good compactifications of $X,Y$ such that $f^{-1}(D_Y)=D$, we also write: $\overline{K}_{X/Y}:=\overline{K}_X-f^*(\overline{K}_Y)=K_{\overline{X}}+D-f^*(K_{\overline{Y}}+D_Y))$.

 \begin{theorem}\label{thmisot}Let $f:X\to Y$ be a smooth, projective,and  birationally isotrivial fibration between smooth connected complex quasi-projective manifolds. Then $\overline{\kappa}(X)=\kappa(X_y)+\overline{\kappa}(Y)$. More generally:
 
  $\kappa(\overline{X}, \overline{K}_{X/Y}+f^*(L))=\kappa(X_y)+\kappa(\overline{Y},L)$, for any line bundle $L$ on $\overline{Y}$.
 \end{theorem}
 
 \begin{proof} $\bullet$ We first give the proof in the simpler case when $X,Y$ are projective, in order to avoid the technicalities due to compactifications. 
 
 By assumption, there is a generically finite base change $v:Y'\to Y$, with $Y'$ smooth, such that $f':X'\to Y'$, which is still smooth, is birational to $X_0:=F\times Y'$, for some $F$ birational to any of the $X_y's$. Let $p:X_0\to Y'$ and $q:X_0\to F$ be the natural projections. Let $u:X'\to X$ be the resulting generically finite map, and $u':X"\to X, \beta:X"\to X_0$ be birational maps with $X"$ smooth, such that $p\circ \beta=f'\circ u':=f":X"\to Y'$. We thus get the following diagram:
 
$ \begin{matrix}
\xymatrix{X"\ar[rd]^{u'}\ar[d]^{\beta}\\
X_0=F\times Y'\ar[rd]^{p}\ar[d]^{q}&X'\ar[rr]^{u}\ar[d]_{f'} && X\ar[d]^{f}\\
F&Y'\ar[rr]^{v} & & Y\\
}
\end{matrix}$

We may, and shall, assume that $X"=X',u'=id_{X'}$. 
 
 By smoothness of $f$, we have: $K_{X'/Y'}=u^*(K_{X/Y})$. This is the point where smoothness enters the picture. The rest is a formal computation. 
 
 We compute $K_{X'/Y'}+(f\circ u)^*(L)=K_{X'/Y'}+(v\circ p \circ\beta)^*(L)$. The second expression gives:
 
 $K_{X'/Y'}+(v\circ p\circ \beta)^*(L)=K_{X'/X_0}+\beta^*(K_{X_0/Y'})+(v\circ p\circ \beta)^*(L)=K_{X'/X_0}+\beta^*(q^*(K_F)+(v\circ p)^*(L))$.
 
 Because $K_{X'/X_0}$ is an effective $\beta$-exceptional divisor, Hartog's theorem implies that $\kappa(X',K_{X'/Y'}+(v\circ p\circ \beta)^*(L))=\kappa(X',\beta^*(q^*(K_F)+(v\circ p)^*(L))= \kappa(X_0, q^*(K_F)+(v\circ p)^*(L))=\kappa(F)+\kappa(Y,L)$.
 
 We next recall: $K_{X'/Y'}+(f\circ u)^*(L)=u^*(K_{X/Y}+f^*(L))$. 
 
 We thus get: $\kappa(X',u^*(K_{X/Y}+f^*(L))=\kappa(X,K_{X/Y}+f^*(L))$.
 
 Comparing the two expressions, we get, as claimed: 
 
 $\kappa(X,K_{X/Y}+f^*(L))=\kappa(F)+\kappa(Y,L)=\kappa(X_y)+\kappa(Y,L)$.
 
 \medskip

 $\bullet$ The quasi-projective case requires additional compactification technicalities. We can, in the situation of Theorem \ref{+}, construct the following diagram of pairs and good compactifications, analog of the previous diagram:

$ \begin{matrix}
\xymatrix{(\overline{X"},D")\ar[rd]^{u'}\ar[d]^{\beta}\\
F\times (\overline{Y'},D_{Y'})\ar[rd]^{p}\ar[d]^{q}&(\overline{X'},D')\ar[rr]^{u}\ar[d]_{f'} && (\overline{X},D)\ar[d]^{f}\\
F&(\overline{Y'},D_{Y'})\ar[rr]^{v} & & (\overline{Y},D_{Y})\\
}
\end{matrix}$

The first step of the construction goes as follows: we construct an initial good compactification $(\overline{X_1},D_1)$ of $X$, then consider the Zariski-open subset $Y$ of the component $\overline{Y_1}$ of the Chow-scheme $Chow(\overline{X_1})$ which parametrises the fibres $X_y$ of $f$ and their limits in $\overline{X_1}$. Let then $\overline{X_2}\subset \overline{X_1}\times \overline{Y_1}$ be the incidence graph of the family parametrised by $\overline{Y_1}$, equipped with the natural projection $f_1:\overline{X_2}\to \overline{Y_1}$. Then $X,Y$ are Zariski open in $\overline{X_2},\overline{Y_1}$, and $f_1$ restricts to our initial $f:X\to Y$ over $X$. We then take suitable birational models $\overline{X}, \overline{Y}, f$ of $\overline{X_2},\overline{Y_1},f_1$ so that $X,Y$ admits the good compactifications $(\overline{X},D)$, $(\overline{Y},D_Y)$, and $f$ restricts to our initial $f$ over $X$, in such a way that $f^{-1}(D_Y)=D$.

Then $f':\overline{X'}\to \overline{Y'}$ is deduced from $f:\overline{X}\to \overline{Y}$ by the generically finite base-change $v:\overline{Y'}\to \overline{Y}$ followed by suitable blow-ups over $(f')^{-1}(D_{Y'})$, so that $f':X':=(\overline{X'}\setminus D')\to Y':=(\overline{Y'}\setminus D_{Y'})$ is deduced from $f:X\to Y$ by the base-change $v:Y'\to Y$ over $Y$.

Next, $u':\overline{X"}\to \overline{X'}$ and $\beta: \overline{X"}\to F\times \overline{Y'}$ are birational, $u:\overline{X'}\to \overline{X}$ and $v:\overline{Y'}\to \overline{Y}$ are generically finite.

This diagram above has then the following properties:

1. $(\overline{X},D)$ is a good compactification of $X$. The other pairs are good compactifications of $Y,X':=(\overline{X'}\setminus D'),Y':=(\overline{Y'}\setminus D_{Y'})$.

2. All maps are maps of pairs. This means that if say $h:(\overline{A},D_A)\to (\overline{B},D_B)$ is such a map, then $\overline{A}, \overline{B}$ are smooth projective, $D_A,D_B$ are simple normal crossings divisors, $h$ is surjective, and $h^{-1}(D_B)=D_A$. We thus also have: $h(D_A)=D_B$, $h(\overline{A}\setminus D_A)=(\overline{B}\setminus D_B)$, and $h^{-1}(\overline{B}\setminus D_B)=(\overline{A}\setminus D_A)$.

Recall that we defined: $K_{\overline{X'}}+D'-u^*(K_{\overline{X}}+D):=\overline{K}_{X'/X}$, as well as: $\overline{K}_{X/Y}:=\overline{K}_X-f^*(\overline{K}_Y)=K_{\overline{X}}+D-f^*(K_{\overline{Y}}+D_Y))$, and similarly for the other maps of pairs. In particular, $\overline{K}_{X'/X}:=\overline{K}_{X'}-u^*(\overline{K}_{X})=K_{\overline{X'}/\overline{X}}-D'+u^*(D)$, where $D'=u^{-1}(D)$.

\begin{lemma}\label{loggenfinite} In the preceding diagram, we have the following properties:

 1. $\overline{K}_{X'/X}=E_u+K^+_{X'/X}$, and $\overline{K}_{Y'/Y}=E_v+K^+_{Y'/Y}$, where $E_u,E_v$ are effective $u$-exceptional and $v$-exceptional divisors respectively, contained in $D'$ and $D_{Y'}$, while $K^+_{X'/X}$ and $K^+_{Y'/Y}$ are the closures of $K_{X'/X}$, and of $K_{Y'/Y}$ in $\overline{X'}$ and $\overline{Y'}$ respectively .
 
 2. $\overline{K}_{X'/Y'}=u^*(\overline{K}_{X/Y})+E_u-(f')^*(E_v)$.

 3. $\kappa(\overline{X'}, \overline{K}_{X'/Y'}+(f')^*(v^*(L)+E_v))=\kappa(\overline{X},\overline{K}_{X/Y}+f^*(L))$ if $L$ is any line bundle on $\overline{Y}$.
\end{lemma}

\begin{proof} 1. Let $E'$ be an irreducible component of $D'$, defined at its generic point $a'$ in local coordinates $(x'_1,\dots,x'_n)$ by the equation $x'_1=0$. Let the image $E$ of $E'$ be contained, near $u(a')$, in the intersection of the components $H_1,\dots H_p$ of $D$ of local equations $x_i=0,i=1,\dots,p$. Then $u^*(H_i)=(x'_1)^{t_i}.f_i(x')$ for integers $t_i>0$ if $i\leq p$, and $t_i=0$ otherwise, and regular functions $f_i$ non-vanishing near $a'$.

Thus $K_{\overline{X'}}+D'$ is generated near $a'$ by: $\frac{dx'_1}{x'_1}\wedge dx'_2\wedge \dots \wedge dx'_n$, while $u^*(K_{\overline{X}}+D)$ is generated by:

 $u^*(\frac{dx_1}{x_1}\wedge\dots \wedge \frac{dx_n}{x_n})=(t_1\frac{dx'_1}{x'_1}+g_1)\wedge\dots\wedge (t_n\frac{dx'_1}{x'_1}+g_n)=\frac{dx'_1}{x'_1}\wedge G_1+G_2$ for regular $k$-forms $g_i,G_1,G_2$ of suitable degrees. Hence the effectivity of $E_u$. 
 
 Let us now show that it is $u$-exceptional, i.e: that we have a contradiction if $E'$ is contained in $E_u$, and if $E$ is a component of $D$, i.e: if $t_i=0$ if $i>1$. Indeed, in this case, for $x'$ generic in $E'$, we may choose $x'_i=u^*(x_i)$ for $i>1$, so that $u(x'_1,\dots x'_n)=(x_1^{t_1}, x_2,\dots x_n)$, and thus: $u^*(K_{\overline{X}}+D)$ is locally generated by $(t_1\frac{dx'_1}{x'_1}\wedge dx'_2\wedge\dots \wedge dx'_n)$, and $u^*(K_{\overline{X}}+D)$ thus generates $K_{\overline{X'}}+D'$ on $E'$.

This proof applies to $\overline{K}_{Y'/Y}$ and $E_v$.

Claim 2 follows from the fact that $K^+_{X'/X}=u^*(K^+_{Y'/Y})$, since $\overline{K}_{X'/Y'}=u^*(K_{X/Y})$ over $X'$.

Claim 3. $\kappa(\overline{X'},\overline{K}_{X'/Y'}+(f')^*(v^*(L)+E_v))=\kappa(\overline{X'},u^*(\overline{K}_{X/Y})+f^*(L)+E_u)=\kappa(\overline{X},\overline{K}_{X/Y}+f^*(L))$, since $E_u$ is $u$-exceptional.
\end{proof}

We can now conclude the proof of Theorem \ref{+} (in generalised form, the given statement is obtained by choosing $L=\overline{K}_{Y}$).

We may, and shall assume that $\overline{X"}=\overline{X'}, D"=D'$. 

From Lemma \ref{loggenfinite}, Claim 3, we just need to check that: $$\kappa(\overline{X'}, \overline{K}_{X'/Y'}+(f')^*(v^*(L)+E_v))=\kappa(F)+\kappa(\overline{Y},\overline{K}_{Y}+L).$$

Since $f'=(p\circ \beta):\overline{X'}\to \overline{Y'}$, we have: $\kappa(\overline{K}_{X'/Y'}+( f')^*(v^*(L)+E_v))=\kappa(\overline{X_0}, \overline{K}_{X_0/Y'}+(p\circ \beta)^*(v^*(L)+E_v))$, where $(\overline{X_0},D_0):=F\times (\overline{Y'},D_{Y'})$, with projections $p:\overline{X_0}\to \overline{Y'}$ and $q: \overline{X_0}\to F$. 

Since $\overline{K}_{X_0/Y'}=q^*(K_F)$, we get: $\kappa(\overline{X_0}, \overline{K}_{X_0/Y'}+(p)^*(v^*(L)+E_v))=\kappa(\overline{X_0}, q^*(K_F)+(p)^*(v^*(L)+E_v))=\kappa(F)+\kappa(\overline{Y'},v^*(L)+E_v)=\kappa(F)+\kappa(\overline{Y'},v^*(L))=\kappa(F)+\kappa(\overline{Y},L)$, since $E_v$ is $v$-exceptional. Hence: $\kappa(\overline{X},\overline{K}_{X/Y}+f^*(L))=\kappa(F)+\kappa(\overline{Y},L)$, as claimed.

Notice that we proved in particular that $\kappa(\overline{X'}, \overline{K}_{X'/Y'}+(f')^*(E_v))=\kappa(\overline{X'}, \overline{K}_{X'/Y'})$, but quite indirectly, by going through $\overline{X'}$ and $\overline{X_0}=F\times \overline{Y'}$. This detour was not needed in the projective case.
\end{proof}

 \begin{corollary}\label{cor2} If $f:X\to Y$ is smooth, projective and birationally isotrivial, with extension $f:(\overline{X},D)\to (\overline{Y},D_Y)$ to good compactifications as above, then: 
 
 $\overline{\kappa}(X)\geq \kappa(\overline{X})\geq \kappa(\overline{X}, K_{\overline{X}}+D-f^*(D_Y))=\kappa(\overline{Y})+ \kappa(X_y)$.
  \end{corollary}
  
  \begin{proof} Apply the preceding result with $L:=-D_Y$, so that $\overline{K_Y}=K_{\overline{Y}}+D_Y$, and observe that $D\leq f^*(D_Y)$ since $D=f^{-1}(D_Y)$.
  \end{proof}

\begin{remark}\label{r2} 1.If there exists a $\Bbb{Q}$-divisor $0\leq \Delta_Y\leq D_Y$ such that $f^*(\Delta_Y)\leq D$, and $\kappa(\overline{X},K_{\overline{X}}+D-f^*(\Delta_Y))=\kappa(\overline{X})$,  the same argument shows that $\kappa(\overline{X})=\kappa(\overline{Y}, K_{\overline{Y}}+(D_Y-\Delta_Y))+\kappa(X_y)$. 

Such a divisor $\Delta_Y$ exists if, for example, the fibres of $f:D\to D_Y$ are irreducible over the generic point of each component of $D_Y$. 

2. One also has the `easy addition' inequality in the other direction: $\kappa(\overline{X})\leq \overline{\kappa}(X)\leq \kappa(X_y)+dim(Y)$, which cannot be improved in general, even for birationally isotrivial fibrations.\end{remark}

 \subsection{Specialness.}

 Recall that, according to the definitions of \cite{Ca04} and \cite{Ca07}, the connected quasi-projective manifold  $Y=\overline{Y}\setminus D_Y$ is special if, for any $p>0$ and any saturated rank-one subsheaf $L\subset \Omega^p_{\overline{Y}}(Log(D_Y))$, one has: $\kappa(\overline{Y},L)<p$. This is independent of the good compactification $(\overline{Y},D_Y)$ of $Y$.

 Quasi-projective curves are either special or hyperbolic. The only special non-projective curves are $\Bbb P_1$ with $d=1,2$ points removed. Special manifolds generalise in higher dimensions rational and elliptic curves. 
 
 Special curves directly generalise in higher dimension as those such that, either $\overline{\kappa}(Y)=0$, or such that $h^0(\overline{Y}, Sym^m(\Omega^p_X(Log(D_Y)))\otimes A)=0, \forall m\geq m_A$, where $A$ is ample on $\overline{Y}$, and $m_A>0$ depends on $A$. The latter ones are the quasi-projective analogues of rationally connected manifolds (and coincide with them when $Y$ is projective, see \cite{Ca17} for details). 
 
 The class of special manifolds is, however, much larger than these two basic examples. In particular, specialness is not determined by $\overline{\kappa}$, unless $\overline{\kappa}\in \{ 0,dim(Y)\}$.
 
 Examples of special projective $Y$ with $dim(Y)=p,\kappa(Y)=k\leq (p-1)$ are the generic smooth divisors of bidegree $(d,p-k+2)$, with $d>k+1$ in $\Bbb P_k\times \Bbb P_{p-k+1}$. Projective, non-special examples with $\kappa(Y)\in \{-\infty,1,\dots,n\}$ are easily constructed as products $Y=Z\times F$, with $Z$ of general type, and $\kappa(F)=0$, or $-\infty$.

 The special quasi projective surfaces are exactly those with $\overline{\kappa}\leq 1$ and $\pi_1$ virtually abelian. Those with $\overline{\kappa}\neq 0,-\infty$ are the ones which fibre over a special curve with special generic orbifold fibres. No such simple description is known, or expected in dimension $3$ or more.
 
 \medskip
 
 Conditionally in the $C_{n,m}^{orb}$-conjecture (Conjecture \ref{cnmorb}, stated below), the special manifolds can be canonically decomposed as towers of fibrations with orbifold fibres having either $\kappa^+=-\infty$, or $\kappa=0$. See \cite{Ca07},\S11 for details.
 
 Special manifolds play an important r\^ole in birational geometry because of the core map, described below.
 
 \subsection{Birational isotriviality and specialness.}

 \begin{proposition} (\cite{AC}, Theorem 5.1) Let $f:X\to Y$ be a surjective and smooth projective morphism with connected fibres between two quasi-projective manifolds. Assume that the fibres have semi-ample canonical bundles, and that $Y$ is special. Then $f$ is birationally isotrivial.
 \end{proposition}
 
 \begin{remark} In \cite{T21}, Theorem 1.1. gives the same conclusion, provided the fibres $X_y$ have good minimal models.\end{remark}

 The text \cite{AC} deals with the more general case when $f$ is given by a regular foliation with smooth but possibly some multiple fibres.The proof in the present submersive case is simpler, and can be directly adapted from the case of canonically polarised manifolds by replacing in \cite{VZ00}, Theorem 1.8, statement (i) by statement (iii), which constructs a big subsheaf $P\subset Sym^m(\Omega^1_{\overline{Y}}(Log(D_Y))$ when $Var(f)=dim(Y)$. Then \cite{V'}, Theorem 1.13, produces the map $\mu:Y\to M$, the  suitable moduli space of polarised manifolds with semi-ample canonical bundle. The arguments of \cite{JK} still work in this context, and construct $\mu^*(det(P))\subset \otimes^{m'} (\Omega^1_{\overline{Y}}(Log(D_Y))$, with $\kappa(\overline{Y},det(P))=Var(f)=dim(Image(\mu))$. Theorem 4.2 of \cite{AC}, based on \cite{CP19}, Corollary 8.7, then shows that $Y$ is not special if $Var(f)>0$.
 
  \begin{example}\label{exvz} An initial example is \cite{VZ01}, Theorem 0.1, which says among other things that if $f:X\to Y$ is smooth projective, with $K_{X_y}$ semi-ample, where $Y$ is a quasi-projective curve, then $f$ is birationally isotrivial if $Y$ is either $\Bbb P_1$ with at most $2$ points deleted, or an elliptic curve, that is: if $Y$ is special.  
 \end{example}

 The following conjecture is the obvious extension of former ones. A positive answer would imply a complete solution of Conjecture 3.4 of \cite{Pop14} by the arguments presented here.
 
 \begin{conjecture} \label{conj}
 
 Let $f:X\to Y$ be a smooth projective morphism with connected fibres between two quasi-projective manifolds. If $K_{X_y}$ is pseudo-effective for $y\in Y$ general, and if $Y$ is special, then $f$ is birationally isotrivial.
 
 \end{conjecture}

 \section{Geometry of orbifold pairs.}

  \subsection{Orbifold pairs, invariants.}
  
  We first briefly recall some definitions and facts on orbifold pairs, and refer to \cite{Ca07}, \cite{CP19}, \cite{Cl} for further details. These notions are applied here only in the two extreme cases when the orbifold divisor $\Delta$ is either $0$ or reduced. The general notion is however needed even then.

  $\bullet$ A smooth orbifold pair $(X,\Delta)$ is a smooth connected complex projective manifold $X$ together with an `orbifold divisor' $\Delta=\sum_Jc_j.D_j$, where $J$ is a finite set, the $D_j's$ are pairwise distinct prime divisors on $X$, the $c_j's$ are rational coefficients in $]0,1]$, and the support $\vert \Delta\vert$ of $\Delta$, which is the union of the $D_j's$, is of simple normal crossings. The coefficients $c_j$ can be written in a unique way as: $c_j=(1-\frac{1}{m_{\Delta}(D_j)})$, for rational numbers greater than $1$, and $+\infty$ if $c_j=1$. The $m_{\Delta}(D_j)'s$ are called the `$\Delta$-multiplicities' of the $D_j's$.
  
  These pairs interpolate between the projective (i.e: $\Delta=0$), and quasi-projective ($\Delta=\vert \Delta\vert$, i.e: $\Delta$ is reduced) cases. We indeed then identify $X$ with $(X,0)$ and $(X,\vert \Delta\vert)$ with $U:=X\setminus \vert \Delta\vert$. Any quasi-projective manifold $U$ is equivalently seen as a smooth orbifold pair $(X,\Delta)$ with reduced boudary $\Delta$, through any of its good compactifications.  
  
  Most of the time we shall give complete definitions only in these two cases ($\Delta=0$, or reduced).
  
$\bullet$ These orbifold pairs come equipped, not only with a canonical divisor $K_X+\Delta$, but also $\Bbb Q$-cotangent bundles and their tensor powers, orbifold morphisms, orbifold birational equivalence. When $\Delta=0$ these invariants are the usual ones, when $\Delta=\vert \Delta\vert$, they coincide with the (`integral' instead of `fractionary') classical log-cotangent bundles $\Omega^1_X(Log(\vert \Delta\vert)$. We say that $(X,\Delta)$ is of general type if $\kappa(X,K_X+\Delta)=dim(X)$.

$\bullet$ We define here the notion of orbifold morphism only in the two situations where either $\Delta=0$, or when $\Delta$ is reduced. In the first case an orbifold morphism is just a usual morphism. In the second case, an orbifold morphism $f:(X,\Delta)\to (Y,\Delta_Y)$ is a morphism $f:X\to Y$ such that $f^{-1}(\Delta_Y)\subset \Delta$, or equivalently, such that $f(X\setminus \Delta)\subset (Y\setminus \Delta_Y)$.
  
  Orbifold morphisms induce maps of $\mathbb Q$-cotangent bundles. This is clear for $\Delta=0$ or reduced.
  
  $\bullet$ If $f:(X',\Delta')\to (X,\Delta)$ is a birational map between $X'$ and $X$, with $\Delta$ and $\Delta'$ reduced, it is said to be an orbifold birational map if and only if $f^{-1}(\Delta)=\Delta'$. It thus induces a birational map between $U':=X'\setminus \Delta'$, and $U:=X\setminus \Delta$. Conversely, any such birational map extends as an orbifold birational equivalence between any of their good compactifications. 
  
 \subsection{Birational neat models.}
  
  If $f:X\to Y$ is a fibration between two projective manifolds, and $X$ is equipped with an orbifold divisor $\Delta$, we simply write:
   `$f:(X,\Delta)\to Y$ is a fibration'.
  
   We then say that $f$ is `$\Delta$-neat' if:
  
  1. The divisor $D_f\subset Y$ over which the fibres of $f$ are singular is of simple normal crossings.
  
  2. $f^{-1}(D_f)\cup \vert \Delta \vert$ is of simple normal crossings.
  
  3. There exists a birational map $u_0:X\to X_0$, with $X_0$ smooth, such that each $f$-exceptional divisor of $X$ is $u_0$-exceptional.
  
  If $f_0:(X_0, \Delta_0)\dasharrow Y_0$ is any rational fibration, one can always find a birational model of $f_0$ which is neat, by suitable modifications of $X_0$ and $Y_0$, first applying flattening to the fibres of $f$, and then desingularising the normalised main component of the fibre product. The orbifold divisor $\Delta$ on the resulting $X$ is defined by taking first the strict transform of $\Delta_0$ by $u_0$, and then adding the exceptional divisors of $u_0$ equipped with coefficients large enough to make $u_0:(X,\Delta)\to (X_0,\Delta_0)$ an orbifold morphism.
  
  When $\Delta$ is reduced, this amounts to equip any $u_0$-exceptional divisor $E\subset X$ with the coefficient $1$ if $u_0(E)\subset \Delta_0$, and with the coefficient $0$ otherwise.

 \subsection{Orbifold base and fibres of a fibration}  
  
   Let $f:(X,\Delta)\to Y$ be a $\Delta$-neat fibration. 
   
   We define the orbifold base $(Y,\Delta_{(f,\Delta)})$ of $(f,\Delta)$ as follows:
 
 For each prime divisor $E\in Y$, $f^*(E)=\sum_kt_k(E,F_k,f).F_k+R$, where $R$ is $f$-exceptional and the $F_k's$ are the pairwise distinct components of $f^{-1}(E)$ which are surjectively mapped onto $E$ by $f$. Define: $m_{f,\Delta}(E):=inf_k\{t_k. m_{\Delta}(F_k)\}$: this is the generic multiplicity of the fibre of $(f,\Delta)$ over $E$. 
 
 Let next $\Delta_{f,\Delta}:=\sum_{E\in Y}(1-\frac{1}{m_{f,\Delta}(E)}).E$: this sum is finite and supported by a simple normal crossing divisor, by the first $2$ neatness conditions above. 
 
 The orbifold pair $(Y,\Delta_{f,\Delta})$ is the `orbifold base' of $(f,\Delta)$.
 
 \medskip

{\bf Orbifold fibres.} Let $f:(X,\Delta)\to Y$ be a fibration between complex projective manifolds. For $y\in Y$ generic, such that the intersection of $X_y$ with each component of $\Delta$ is transversal, we define $\Delta_{X_y}$ as the intersection $X_y\cap \Delta$, the coefficient on each component being the same one as in $\Delta$.This is the generic (smooth) orbifold fibre of $(f,\Delta)$.

\medskip

  The following elementary property is used in step 5 of the proof of Theorem \ref{+}.
  
   \begin{lemma}\label{>}Let $f:X\to Y$ and $h:Y\to Z$ be two fibrations between projective manifolds. Assume $f$ to be neat, and $h:(Y,\Delta_f)\to Z$ to be $\Delta_f$-neat. Then $\Delta_{h\circ f}\geq \Delta_{h,\Delta_f}$ (i.e: the difference is effective).
   
   In particular, if $(Z,\Delta_{h,\Delta_f})$ is of general type, so is $(Z,\Delta_{h\circ f})$.
\end{lemma}

\begin{proof} Let $E\subset Z$ be a prime divisor, and $h^*(E)=\sum_kt_k.F_k+R_h$, with $R_h$ an $h$-exceptional divisor, and $h(F_k)=E,\forall k$. We thus have: $(h\circ f)^*(E)=f^*(\sum_kt_k.f^*(F_k))+f^*(R_h)$. The last term is $(h\circ f)$-exceptional. If $G\subset X$ is an irreducible divisor such that $f(G)=F_k$ for some $k$ appearing in the first term, its multiplicity in $\Delta_{h\circ f}$ is at least $t_k. m_{\Delta_f}(F_k)$, since $m_{\Delta_f}(F_k):=inf_{G}\{s_G\}$, where $f^*(F_k)=\sum_{r} s_r.G_r+R_{f,k}$, where $R_{f,k}$ is $f$-exceptional, and $r$ is the set of prime divisors $G_r\subset X$ which are mapped onto $F_k$ by $f$. Hence the claim.
\end{proof}

A slightly more general version is given in Lemma \ref{>'}.

  \subsection{Kodaira dimension of the orbifold base.}

  \begin{theorem}(\cite{Ca07}, 5.3) Let $f:(X,\Delta)\dasharrow Y$ be a rational fibration, with $Y$ smooth of dimension $p>0$. 
  
  Let $L_{f,\Delta}\subset \Omega^p(X,\Delta)$ be the saturation of $f^*(K_{Y})$ in $\Omega^p(X,\Delta)$. Let $f:(X',\Delta')\to Y'$ be any $\Delta'$-neat birational model of $f$. 
  
  Then: $\kappa(X, L_{f,\Delta})=\kappa(Y',K_{Y'}+\Delta_{f',\Delta'})=\kappa(X',L_{f',\Delta'})$. \end{theorem}
  
   Instead of defining what is meant here in general by the saturation in $\Omega^p(X,\Delta)$, let us just recall that in the two extreme cases considered here ($\Delta=0$, and $\Delta$ reduced), $\Omega^p(X,\Delta)$ coincides respectively with the classical $\Omega^p_X$ and with $\Omega^p_X(Log(\Delta))$.

  \begin{corollary} $(X,\Delta)$ is special if and only if, for any  of any rational fibration $f:(X,\Delta)\dasharrow Y$, and any of its $\Delta'$-neat birational models $f':(X',\Delta')\to Y'$, the orbifold base $(Y',\Delta_{f',\Delta'})$ is {\bf not} of general type.  \end{corollary}

 \subsection{Orbifold version of Viehweg's additivity.}

We have next the following orbifold extension of Viehweg's additivity theorem:

\begin{theorem}\label{Vieh} (\cite{Ca07}, 7.3, \cite{Ca04},4.13) Let $(X,\Delta)$ be a smooth orbifold pair, and $f:X\to Y$ be a $\Delta$-neat fibration. Assume that its orbifold base $(Y,\Delta_{f,\Delta})$ is of general type. Then $\kappa(X,\Delta)=\kappa(X_y,\Delta_{X_y})+dim(Y)$.
\end{theorem}

When $\Delta=0=\Delta_{f}$, this is Viehweg's additivity theorem. When $\Delta=\vert\Delta\vert$, it strengthens \cite{kaw'},Theorem 30, and \cite{Mae}, Corollary 2.

\medskip

Viehweg's orbifold additivity solves in this particular case the following orbifold version $C_{n,m}^{orb}$ of the $C_{n,m}$ conjecture.

\begin{conjecture}\label{cnmorb}(\cite{Ca07}, Conjecture 7.1) Let $f:(X,\Delta)\to Y$ be a $\Delta$-neat fibration between projective manifolds.

Then: $\kappa(X,K_X+\Delta)\geq \kappa(X_y,K_{X_y}+\Delta_y)+\kappa(Y,K_Y+\Delta_{f,\Delta})$.
\end{conjecture}
 
 \subsection{The core map.}

 The notion of special orbifold pair permits to functorially\footnote{Into the category of dominant rational maps.}decompose by one single fibration any smooth projective orbifold $(Y,\Delta_Y)$ into its `special' and `general type' parts. We formulate here this statement only when $\Delta$ is reduced (or zero).

 \begin{theorem}(\cite{Ca07}, 10.1) Let $Y$ be a quasi-projective manifold. There is then a unique almost holomorphic\footnote{This means that its generic fibre does not meet its indeterminacy locus.} fibration $c_Y:Y\dasharrow C_Y$ such that its general fibre is special, and the orbifold base $(C',\Delta_{c'})$ of any of its compactified neat models $(\overline{Y},\Delta_Y)$ is of general type. 
 
  More precisely: one can choose a good compactification $(\overline{Y},D_Y)$ of a modification of $Y$ such that the core map is represented by a $D_Y$-neat fibration $c_Y:(\overline{Y},D_Y)\to (\overline{C},D_C)$. 
 \end{theorem}
 
 The two extreme cases are when $Y$ is either special (and $C$ a point), or of general type (and $C=Y$).

 \begin{remark} Assume that $f:X\to Y$ is smooth, projective, with $X,Y$ quasi-projective, and that $K_{X_y}$ is semi-ample. Then $dim(Y)\geq Var(f)\geq 0$, where $Var(f)$ is the generic rank over $Y$ of $R^1f_*(T_{X/Y})$. It follows from \cite{AC}, Theorem 5.1, that $Var(f)\leq dim(C_Y)$.
 \end{remark}

 \section{The Quasi-Projective Case.}\label{open}

We now prove Theorem \ref{+} in the quasi-projective case. The steps are the same as in the projective case, we just need to give the additional details required about the compactifications. 

We thus have $f:X\to Y$ smooth projective as in the statement of theorem \ref{+}. We extend $f$ to good compactifications $f:(\overline{X},D)\to (\overline{Y}, D_Y)$ and may assume that the core map $c_Y:(\overline{Y},D_Y)\to (\overline{C},D_C)$ is regular and $D_Y$-neat.

Since the fibres of $f:(\overline{X},D)\to (\overline{Y},D_Y)$ are smooth, projective, and the boundary divisors $D,D_Y$ are reduced, the orbifold base of $f:(\overline{X},D)\to (\overline{Y},D_Y)$ is just $(\overline{Y},D_Y)$.

We are thus in position to apply the Lemma \ref{>'} below, general version of Lemma \ref{>}, to $(\overline{X},D)$, $c_Y$ and $\overline{C}$ in place of $(X,\Delta)$, $g$ and $Z$. This Lemma shows that: $\Delta_{c_Y\circ f,D}\geq \Delta_{c_Y,D_Y}$, and so the left-hand term is of general type since $\Delta_{c_Y,D_Y}$ is. From the orbifold version of Viehweg's additivity (i.e: Theorem \ref{Vieh}), we thus get:

$\overline{\kappa}(X)=\overline{\kappa}(X_c)+dim(C)$, and:

$\overline{\kappa}(Y)=\overline{\kappa}(Y_c)+dim(C)$, and thus:

$\overline{\kappa}(X)-\overline{\kappa}(Y)=\overline{\kappa}(X_c)-\overline{\kappa}(Y_c)=\kappa(X_y)$, by the proof of Theorem \ref{+} in the birationally isotrivial case given in \S.\ref{birisot}.

   \begin{lemma}\label{>'}Let $f:(X,\Delta)\to Y$ and $h:Y\to Z$ be two fibrations between projective manifolds. Assume $f$ to be $\Delta$-neat, and $h:(Y,\Delta_{f,\Delta})\to Z$ to be $\Delta_{f,\Delta}$-neat. Then $\Delta_{h\circ f,\Delta}\geq \Delta_{h,\Delta_{f,\Delta}}$ (i.e: the difference is effective).
   
   In particular, if $(Z,\Delta_{h,\Delta_{f,\Delta}})$ is of general type, so is $(Z,\Delta_{h\circ f,\Delta})$.
\end{lemma}

The proof is the same as the proof of Lemma \ref{>}, by adding the contribution of $\Delta$. Just the notations are more involved.

\begin{proof} Let $E\subset Z$ be a prime divisor, and $h^*(E)=\sum_kt_k.F_k+R_h$, with $R_h$ an $h$-exceptional divisor, and $h(F_k)=E,\forall k$. We thus have: $(h\circ f)^*(E)=f^*(\sum_kt_k.f^*(F_k))+f^*(R_h)$. The last term is $(h\circ f)$-exceptional. If $G\subset X$ is an irreducible divisor such that $f(G)=F_k$ for some $k$ appearing in the first term of the preceding sum, its multiplicity in $\Delta_{h\circ f,\Delta}$ is at least $t_k. m_{\Delta_{f,\Delta}}(F_k)$, since $m_{\Delta_{f,\Delta}}(F_k):=inf_{r}\{s_{G_r}.m_{\Delta}(G_r)\}$, where $f^*(F_k)=\sum_{r} s_r.G_r+R_{f,k}$, where $R_{f,k}$ is $f$-exceptional, and $r$ is the set of prime divisors $G_r\subset X$ which are mapped onto $F_k$ by $f$. Hence the claim.
\end{proof}

\end{document}